\documentclass{article}
\usepackage{graphicx} 

\title{A Characterization of the Cumulants\\ as Continuous Moment-Based Statistics}
\author{Sofia de la Cerda\thanks{{University of Chicago}. \texttt{sofiasl@uchicago.edu} }
}
\date{}

\usepackage[margin=1in]{geometry}
\usepackage{setspace}
\usepackage{amsfonts,amsmath,amssymb,amsthm}
\usepackage[backref=true,style=alphabetic,maxnames=4]{biblatex}
\usepackage[bookmarks=true,
    bookmarksopen=true,
    bookmarksopenlevel=3,
    bookmarksnumbered=true,
    hidelinks
]{hyperref} 

\usepackage{bbm}
\usepackage{mathtools}

\newtheorem{theorem}{Theorem}[section]

\newtheorem{proposition}[theorem]{Proposition}

\usepackage[LGR,T1]{fontenc}
\usepackage[english]{babel}
\usepackage{alphabeta}
\usepackage{csquotes}

\newcommand{\R}{\mathbb{R}}
\newcommand{\C}{\mathbb{C}}

\newcommand{\N}{\mathbb{N}}
\newcommand{\E}{\mathbb{E}}

\begin{document}

\maketitle

\begin{abstract}
Cumulants are classical statistics associated with a random variable, defined as polynomial functions of its moments and distinguished by their additivity under convolution of distributions. Statistic is the name given to a function of a random variable, and a moment-based statistic is one that depends only on the moments $(\E[X^n])_{n\in \N}$. We prove a converse: any statistic depending continuously on finitely many moments and additive for independent sums must be a linear combination of cumulants.

The proof uses an algebraic reformulation of the problem via the Hurwitz product and a linearizing change of coordinates. This result also follows from the more general theorem of Mattner \cite{mattner}, but our approach is elementary and self-contained. 
\end{abstract}

\section{Introduction}

The cumulants $(\kappa_n)_{n\in \N}$ are real functions defined on the space of random variables with finite $n$-th moment, i.e., $\kappa_n:L^n\to\R$. They are so named because they are additive in the sense that $\kappa_n(X+Y)=\kappa_n(X)+\kappa_n(Y)$ for independent random variables $X,Y$. They are defined by the relation
$$\log\left(\E\left[e^{zX}\right]\right)=\sum_{\ell=1}^n\frac{z^\ell}{\ell!}\kappa_{\ell}(X) +o(z^n)\quad\textup{as }z\to 0$$
for $z\in i\R$. Note that this does not necessarily hold for $z\in \C$. For concreteness, we note that $\kappa_1(X)=\E[X]$ coincides with the expectation, $\kappa_2(X)=\E[X^2]-\E[X]^2$ is the variance, and $\kappa_3(X)=\E[X^3]-3\E[X^2]\E[X]+2\E[X]^3$ is finally a new function, and $$\kappa_4(X)=\E[X^4]-4\E[X^3]\E[X]-3\E[X^2]^2+12\E[X^2]\E[X]^2-6\E[X]^4.$$ 

For a normal distribution, the cumulants of order $3$ and higher all vanish, a property that characterizes normality \cite{lukacs}. This vanishing property motivates thinking of higher-order cumulants as measures of deviation from normality. In particular $\kappa_3$ is a measure of asymmetry, and $\kappa_4$ is a measure of tailedness. The dilation-invariant quotient $\kappa_3/\kappa_2^{3/2}$ is known as skewness, and $\kappa_4/\kappa_2^2$ is known as kurtosis (in some texts excess kurtosis), from the Greek \textkappa\textupsilon\textrho\texttau\'\textomicron\textvarsigma \hphantom{ }meaning convex or curved. 

For additivity let $X,Y$ be independent. Then $e^{zX}, e^{zY}$ are independent, then
\begin{align*}
    \sum_{\ell=1}^n\frac{z^\ell}{\ell!}\kappa_{\ell}(X+Y)+o(z^n)&=\log\left(\E\left[e^{z(X+Y)}\right]\right)\\
    &=\log\left(\E\left[e^{zX}\right]\E\left[e^{zY}\right]\right)\\
    &=\log\left(\E\left[e^{zX}\right]\right)+\log\left(\E\left[e^{zY}\right]\right)\\
    &=\sum_{\ell=1}^n\frac{z^\ell}{\ell!}(\kappa_{\ell}(X)+\kappa_\ell(Y))+o(z^n)
\end{align*}
thus $\kappa_{\ell}(X+Y)=\kappa_{\ell}(X)+\kappa_{\ell}(Y)$ for every $\ell\in [n]=\{1, \dots, n\}$.

To relate the cumulants to the moments, recall Faà di Bruno's formula
$$\frac{d^n}{dx^n}f(g(x))=\sum_{\ell=1}^n f^{(\ell)}(g(x))B_{n,\ell}(g^\prime(x), \dots,g^{(n-\ell+1)}(x))$$
where $B_{n,\ell}$ is an incomplete Bell polynomial---defined as
\begin{align*}
    B_{n,\ell}(x_1, \dots, x_{n-\ell+1})=\sum_{\substack{j_1, \dots, j_{n-\ell+1}\in \N\\ j_1+\dots+j_{n-\ell+1}=\ell \\ j_1+2j_2+\dots+(n-\ell+1)j_{n-\ell+1}=n}}
    \frac{n!}{j_1!\dots j_{n-\ell+1}!}\left(\frac{x_1}{1!}\right)^{j_1}\left(\frac{x_2}{2!}\right)^{j_2}\dots&\left(\frac{x_{n-\ell+1}}{(n-\ell+1)!}\right)^{j_{n-\ell+1}}.
\end{align*}

An application of the formula for $f=\exp$ and $g(z)=\log(\E[e^{zX}])$ gives
\begin{align*}
    \frac{d^n}{dz^n}\Bigg\rvert_{z=0}\E[e^{zX}]&=\sum_{\ell=1}^n e^{g(0)}B_{n,\ell}(\kappa_1(X), \dots, \kappa_{n-\ell+1}(X))\\
    \E[X^n]&=\sum_{\ell=1}^n B_{n,\ell}(\kappa_1(X), \dots, \kappa_{n-\ell+1}(X)),
\end{align*}
i.e., the moments of $X$ can be expressed as polynomials in its cumulants. The polynomials $B_n=\sum_{\ell=1}^n B_{n,\ell}$ thus give the moments as functions of the cumulants, and they are called exponential Bell polynomials.

Another application of the formula for $f=\log$ and $g(z)=\E[e^{zX}]$ gives
\begin{align*}
    \frac{d^n}{dz^n} \Bigg\rvert_{z=0}\log\left(\E\left[e^{zX}\right]\right)&=\sum_{\ell=1}^n (-1)^{\ell-1} (\ell-1)!\frac{1}{g(0)^\ell}B_{n,\ell}(\E[X], \dots, \E[X^{n-\ell+1}])\\
    \kappa_n(X)&=\sum_{\ell=1}^n (-1)^{\ell-1} (\ell-1)!B_{n,\ell}(\E[X], \dots, \E[X^{n-\ell+1}]),
\end{align*}
i.e., the cumulants are expressed as polynomials in the moments via the logarithmic Bell polynomials $\overline B_n=\sum_{\ell=1}^n (-1)^{\ell-1}(\ell-1)! B_{n,\ell}$. This is a particular case of the relation of formal power series
\begin{equation}\label{logbell}
    \log\left(1+\sum_{\ell=1}^n \frac{x_\ell}{\ell!}z^\ell\right)=\sum_{\ell=1}^n\frac{\overline B_{\ell}(x_1, \dots, x_\ell)}{\ell!}z^\ell+O(z^{n+1}).
\end{equation}

It is remarkable that the function $(x_1, \dots, x_n)\mapsto (B_1(x_1), \dots, B_n(x_1,\dots, x_n))$ is (entry-wise) polynomial with a polynomial inverse $(x_1, \dots, x_n)\mapsto (\overline B_1(x_1), \dots, \overline B_n(x_1,\dots, x_n))$. 

We are now ready to state the main theorem. Formally, it is not an original result as it follows from \cite[Theorem 1.11]{mattner}, however, to the best of the author's knowledge, this specific formulation and its proof are new. The historical background is further discussed in \autoref{literature}.
\begin{theorem}\label{main}
    Let $F:L^n\to\R$ be such that there exists a continuous function $f:\R^n\to \R$ with $F(X)=f(\E[X], \dots, \E[X^n])$. Suppose that $F$ is additive, i.e., for all independent $X,Y\in L^n$, $F(X+Y)=F(X)+F(Y)$. Then $F$ is a linear combination of cumulants of order at most $n$, equivalently $f$ is (in the appropriate domain) a linear combination of logarithmic Bell polynomials.
\end{theorem}

Requiring that $F$ depend only on the moments is not that significant a restriction, since under mild conditions a distribution is uniquely determined by its moments. Moreover, the hypotheses of the theorem are weak---continuity and additivity---and its consequences are strong---polynomial and even more rigidly spanned by a finite family given explicitly. 

The fundamental observation in the proof is that addition of independent random variables, at the level of moments, is an algebraic operation (the Hurwitz product) involving binomial expansions. After encoding additivity in these terms, we find a system of coordinates that transforms this product into ordinary addition in $\R^n$. The problem then reduces to solving a Cauchy functional equation under hypotheses that imply linearity. The cumulants appear as a basis associated with this linearization.

Before proceeding, define the set of moment sequences 
$$\Omega_n=\big\{(\E[X], \dots, \E[X^n]):X\in L^n\big\}$$
so we can outline the proof:
\begin{enumerate}
    \item Prove that $\Omega_n$ has non-empty interior;
    \item Define the Hurwitz product $*$ on $\R^n$ and observe that the additivity hypothesis of the main theorem is equivalent to $f(x*y)=f(x)+f(y)$;
    \item The Hurwitz product gives $\R^n$ the structure of an abelian group which is isomorphic to the usual $\R^n$ via the function $\varphi$, i.e., $\varphi(x+y)=\varphi(x)*\varphi(y)$;
    \item $f\circ\varphi$ extends to a linear functional on $\R^n$;
    \item The compositions $\overline B_n\circ \varphi$ span the space of linear functionals. 
\end{enumerate}

\section{Previous Work}\label{literature}

The cumulants were introduced by Thiele, and the reader can find a translated version of his paper in \cite{thiele}. Rota \cite{Rota2001} also credits him with the following characterization of the cumulants:

\begin{quote}
    Sometime in the last [19th] century, the Danish statistician Thiele observed that the variance of a random variable, namely $\text{Var}(X)=\mathbb E[X^2]-\mathbb E[X]^2$, possesses notable properties, to wit:
    \begin{enumerate}
        \item It is invariant under translation: $\text{Var}(X+c)=\text{Var}(X)$ for any constant $c$.
        \item If $X$ and $Y$ are independent random variables, then $\text{Var}(X + Y) = \text{Var}(X) + \text{Var}(Y)$.
        \item $\text{Var}(X)$ is a polynomial in the moments of the random variable $X$.
    \end{enumerate}
    
    He then proceeded to determine all nonlinear functionals of a random variable which have the same properties.
\end{quote}

Thiele's works predate modern notation, and we could not find a modern statement of the polynomial characterization in the literature. His characterization (as described by Rota) is similar to ours, with the main difference being that he considers only functions of the form $f(\E[X], \dots, \E[X^n])$ for polynomial $f$, whereas we considered any continuous $f$.

Another far-reaching characterization was found by Mattner \cite{mattner}. He gives the space of probability measures with all moments finite a topology, and proves that functions that are additive and continuous with respect to that topology are linear combinations of cumulants. He chooses a strong topology, so continuity is a weak assumption on $f$, and in fact our result can be obtained as a particular case of his. His proof uses functional and Fourier analytic techniques. The present work may thus be viewed as an intermediate result between Thiele's and Mattner's characterizations.

\section{The Hamburger Moment Problem}

The Hamburger moment problem is: which sequences $x=(x_\ell)_{\ell\in [n]}$ are moment sequences, in the sense that there exists a random variable $X$ such that $\E[X^\ell]=x_\ell\:\forall\:\ell\in [n]$, equivalently $x\in\Omega_n$? Is this random variable unique up to distribution?

The first question has a rather simple answer, and the second one is beyond the scope of this paper, but very interesting. Its answer is negative in general, but positive under mild restrictions (e.g., Carleman's condition) under which the corresponding random variables are characterized by their moments. In this context, the hypothesis of the main theorem that $F$ may only depend on the moments appears less restrictive than it does at first glance.

\begin{theorem}[{\cite[Theorem 9.27 and Theorem 9.32]{momentproblem}}]
    A sequence $(x_\ell)_{\ell\in [n]}$ is a moment sequence if and only if it can be extended to $(x_\ell)_{\ell\in [m]}$ where $m=2\lfloor n/2\rfloor+2$ such that
    the Hankel matrix associated with it
    $$\begin{bmatrix}
        1 & x_1 & x_2 &\cdots  & x_{m/2}\\
        x_1 & x_2 & \cdots & \cdots & x_{m/2+1}\\
        x_2 & \vdots & \ddots & &\vdots\\
        \vdots & \vdots & & \ddots & \vdots \\
        x_{m/2} & x_{m/2+1} & \cdots & \cdots & x_{m}
    \end{bmatrix}$$
    is positive semidefinite. That is, if $n$ is odd there exists $x_{n+1}\in\R$ or, if $n$ is even there exist $x_{n+1},x_{n+2}\in\R$ such that the above matrix is positive semidefinite.
\end{theorem}

For any moment sequence, the determinant of the second leading minor is $x_2-x_1^2=\textup{Var}(X)$, so we must have $\textup{Var}(X)\geq 0$. For a length-two sequence $(x_1, x_2)$, this is also sufficient as it can be achieved by, e.g., a normal distribution $\mathcal N(x_1, x_2-x_1^2)$, and the corresponding choice $x_3=3x_1x_2-2x_1^3$ and $x_4=3x_2^2-2x_1^4$ makes the Hankel matrix have determinant $2(x_2-x_1^2)^3\geq0$, so by Sylvester's criterion it is positive semidefinite.

Knowing this theorem, we can prove the main result of this section.

\begin{proposition}
    The set $\Omega_n$ contains a neighborhood of $x=(\ell!)_{\ell\in [n]}$.
\end{proposition}

\begin{proof}
    To apply Sylvester's criterion, compute the determinant of the leading minors 
    $$\det\begin{bmatrix}
        1 & 1! & \cdots & k!\\
        1! & 2! & \cdots & (k+1)!\\
        \vdots & \vdots & \ddots & \vdots \\
        k! & (k+1)! & \cdots & (2k)!
    \end{bmatrix}=\prod_{\ell=1}^k \ell!^2 $$
    \cite[Example 2]{JUNOD}. These determinants (for $k=1,\dots, \lfloor n/2\rfloor+1$) are all strictly positive, so by continuity of $\det$ there is a neighborhood of $x$ where they are all positive. For any element of this neighborhood, we can apply Sylvester's criterion to prove that the associated Hankel matrix is positive semidefinite, thus there exists a random variable with the corresponding moments, thus it is an element of $\Omega_m\subset\R^m$. The projection of $\Omega_m$ to the first $n$ coordinates then contains an open neighborhood of $x$.
\end{proof}

This vector $(\ell!)_{\ell\in [n]}$ is a natural choice as the moment sequence of a standard exponential distribution.

\section{The Hurwitz Product}

In the previous section, we discussed the map that takes a random variable to its moment sequence. We now define the Hurwitz product $*$, which corresponds to the sum of independent random variables under this map. More precisely, let $x=(\E[X^\ell])_{\ell \in [n]}$ and $y=(\E[Y^\ell])_{\ell \in [n]}$, then $x*y$ is the moment sequence of $X+Y$.

We might think of the Cauchy product of two sequences $\left(\sum_k x_{k}y_{\ell-k}\right)_{\ell\in [n]}$ as the operation that gives the coefficients of the product of two generating functions $\sum_\ell x_\ell z^\ell\sum_\ell y_\ell z^\ell$. We are taking $x=(x_1, \dots, x_n)$ and $x_0=1$ for convenience, and similarly for all other sequences. Likewise, the Hurwitz product 
$$x*y=\left(\sum_{k=0}^\ell \binom{\ell}{k}x_ky_{\ell-k}\right)_{\ell\in [n]}$$
gives the sequence of coefficients of the product of exponential generating functions
\begin{equation}\label{prod}
    \sum_{\ell=0}^n \frac{x_\ell}{\ell!}z^\ell\sum_{\ell=0}^n \frac{y_\ell}{\ell!}z^\ell=\sum_{\ell=0}^n \frac{(x*y)_\ell}{\ell!}z^\ell+O(z^{n+1}).
\end{equation}
For example,
$$(x_1, x_2, x_3)*(y_1, y_2, y_3)=\left(x_1+y_1, x_2+2x_1y_1+y_2, x_3+3x_2y_1+3x_1y_2+y_3\right).$$
This operation is commutative despite its name. 

Recall the additivity hypothesis of the main theorem: $F(X+Y)=F(X)+F(Y)$. It is now clear that it is equivalent to $f(x*y)=f(x)+f(y)$ for all $x,y\in\Omega_n$.

In fact, the Hurwitz product gives $\R^n$ the structure of an abelian Lie group, but we will not need that fact. What we will need is that $\Omega_n$ is a sub-semigroup of $(\R^n, *)$, which is easy to verify since for $x,y\in\Omega_n$ we can find random variables $X,Y$ (which can be taken to be independent)  with moment sequences $x,y$ then $X+Y$ is a random variable whose moments are $x*y$ which is an element of $\Omega_n$.

\section{Linearization}

William Stein said that mathematics is the art of reducing any problem to linear algebra, and this paper is no exception. We now define a function $\varphi$ such that $f\circ\varphi$ is linear.
\begin{align*}
    \varphi:\R^n&\to \R^n\\
    x&\mapsto u\quad \textup{where } \prod_{k=1}^n (1+z^k)^{x_k}=1+\sum_{\ell=1}^n \frac{u_\ell}{\ell!}z^\ell.
\end{align*}
We omit terms of order $O(z^{n+1})$. 

\begin{proposition}
    The function $\varphi$ is a group isomorphism to $(\R^n, *)$ from the usual $\R^n$.
\end{proposition}

\begin{proof}
    Indeed, let $x,y\in\R^n$ and $\varphi(x)=u$ and $\varphi(y)=v$, then
\begin{align*}
    1+\sum_{\ell=1}^n \frac{\varphi(x+y)_\ell}{\ell!}z^\ell&=\prod_{k=1}^n(1+z^k)^{x_k+y_k}=\prod_{k=1}^n(1+z^k)^{x_k}\prod_{k=1}^n(1+z^k)^{y_k}\\
    &=\left(1+\sum_{\ell=1}^n \frac{u_\ell}{\ell!}z^\ell\right)\left(1+\sum_{\ell=1}^n \frac{v_\ell}{\ell!}z^\ell\right)
\end{align*}
by \autoref{prod}, $\varphi(x+y)=\varphi(x)*\varphi(y)$.

Furthermore if $\varphi(x)=\varphi(y)$ we have
\begin{align*}
    \prod_{k=1}^n (1+z^k)^{x_k}&=\prod_{k=1}^n (1+z^k)^{y_k}\\
    \prod_{k=1}^n (1+z^k)^{x_k-y_k}&=1+O(z^{n+1})
\end{align*}
from which we can extract the $z$ coefficient, implying $x_1=y_1$ then the $z^2$ coefficient implies $x_2=y_2$, etc. Inductively $x=y$, i.e., $\varphi$ is injective.

Finally, observe that 
$$\varphi(x)=\left(x_1, 2!\binom{x_1}{2}+2x_2, 3!\binom{x_1}{3}+6x_1x_2+6x_3, \dots\right)$$
in particular the $\ell$-th coordinate $\varphi(x)_\ell$ is something that only depends on $x_1, \dots, x_{\ell-1}$ plus $\ell!x_\ell$. So for any $u\in\R^n$ we can find its preimage by recursion on the coordinates: $x_1=u_1$, and $x_2=\frac 12 \left(u_2-2\binom{x_1}{2}\right)$, $x_3=\frac{1}{3!}(u_3-\dots)$, etc. That is, $\varphi$ is surjective. 
\end{proof}

The reason why we need this isomorphism is to state this double equality which is key to the entire proof
$$f\circ\varphi(x+y)=f(\varphi(x)*\varphi(y))=f\circ\varphi(x)+f\circ\varphi(y)\quad\textup{for all }x,y\in \varphi^{-1}(\Omega_n)!$$
In words: $f\circ\varphi$ is additive. Now, since $\varphi$ is continuous, the inverse image of a ball $\varphi^{-1}(B)$ is open, and since it is surjective, it is not empty so a standard result from algebra proves linearity. The additivity condition $g(x+y)=g(x)+g(y)$ is equivalent to the Cauchy functional equation, which under these conditions has a unique linear solution.

\begin{proposition}
    Let $S\subset \R^n$ be  a sub-semigroup with non-empty interior and $g:S\to \mathbb R$ be a continuous additive function, i.e., $g(x+y)=g(x)+g(y)$ for all $x,y\in S$. Then $g$ extends to a unique linear functional on $\R^n$.
\end{proposition}
\begin{proof}
    Define $\overline{g}:S-S\to \R$ by $\overline g(x-y)=g(x)-g(y)$, where $S-S=\{x-y:x,y\in S\}$. To verify that this is well-defined, consider two representatives of the same point $x_1-y_1=x_2-y_2$ then
    \begin{align*}
        x_1+y_2&=x_2+y_1\\
        g(x_1+y_2)&=g(x_2+y_1)\\
        g(x_1)+g(y_2)&=g(x_2)+g(y_1)\\
        g(x_1)-g(y_1)&=g(x_2)-g(y_2).
    \end{align*}
    Furthermore, for any point $p\in S-S$ write it as $p=x-y$ thus $p+y=x$ which implies that any additive extension $g'$ in $S-S$ satisfies $g'(p+y)=g'(x)$ thus $g'(p)=g'(x)-g'(y)=g(x)-g(y)=\overline g(p)$, that is, $\overline g$ is the unique extension of $g$ to an additive function in $S-S$.
    Since $S$ contains an open ball $B_\epsilon(v)$, $S-S$ contains $B_\epsilon(v)-B_\epsilon(v)=B_{2\epsilon}(0)$, thus $S-S$ is a subgroup of $\R^n$ that contains a neighborhood of the origin. For any $x\in\R^n$ we can write $x=\frac{x}{m}+\dots+\frac{x}{m}$ where $m$ is large enough that $\frac{x}{m}\in S-S$, thus $S-S=\R^n$. Since $\overline g$ is continuous and additive on $\R^n$ an application of the standard Cauchy functional equation completes the proof.
\end{proof}

From now on, we will consider $f\circ\varphi$ to be its linear extension to $\R^n$. We must now prove that the functions $\overline{B}_\ell\circ\varphi$ are linear and find their closed form. Let $u=\varphi(x)$ and compute

\begin{align*}
        1+\sum_{\ell=1}^n \frac{u_\ell}{\ell!}z^\ell&=\prod_{k=1}^n(1+z^k)^{x_k}+O(z^{n+1})\\
        \log\left(1+\sum_{\ell=1}^n \frac{u_\ell}{\ell!}z^\ell\right)&=\sum_{k=1}^n x_k\log(1+z^k)+O(z^{n+1})\\
        \sum_{\ell=1}^n \frac{\overline B_\ell(u)}{\ell!}z^\ell&=\sum_{k=1}^n x_k\log(1+z^k)+O(z^{n+1}).\\
\end{align*}
Wherein we used the definition of $\varphi$, took a logarithm, and applied \autoref{logbell}. Now we use the power series for the logarithm
$$[z^n]\log(1+z^d)=\begin{cases}
    \frac{(-1)^{n/d+1}}{n/d} & d\mid n\\
    0 & d\nmid n.
\end{cases}$$
Where prepending $[z^k]$ to a power series denotes the operation that gives its $z^k$ coefficient. Finally extract the $z^n$ coefficient,
\begin{align}
        \overline B_n(u)=\overline B_n\circ\varphi(x)&=n![z^n]\sum_{k=1}^n x_k\log(1+z^k)\\
        &=n!\sum_{d\mid n}x_d[z^n]\log(1+z^d)\\
        &=(n-1)!\sum_{d\mid n}(-1)^{n/d+1}d x_d.\label{cumulantlinear}
\end{align}

The attentive reader will recognize the right-hand side as a Dirichlet convolution, a remark we will make no use of. 

\section{Proof of Theorem \ref{main}}

\begin{proof}
    Consider the stacking $\left(\overline B_1\circ\varphi, \dots, \overline B_n\circ\varphi\right)$ which by \autoref{cumulantlinear} is represented by an $n\times n$ triangular matrix with non-vanishing diagonal entries, and in particular is invertible. Therefore $\left(\overline B_\ell\circ\varphi\right)_{\ell\in [n]}$ is linearly independent, so their span is the entire space of linear functionals. Hence, there exist constants $c_1, \dots, c_n$ such that  
    $$f\circ\varphi=\sum_{\ell=1}^n c_\ell\cdot \overline{B}_\ell\circ\varphi.$$
    Now apply $\varphi^{-1}$ on the right to obtain
    $$f=\sum_{\ell=1}^n c_\ell\cdot \overline{B}_\ell,$$
    i.e., $F=\sum_\ell c_\ell \kappa_\ell$.
\end{proof}

\section*{Acknowledgments}

We thank Roger Lee for helpful discussions and Madhur Tulsiani for drawing attention to reference \cite{JUNOD}. We thank ChatGPT 5.2 and Claude Sonnet 4.5 for their help in writing this manuscript.

\phantomsection
\addcontentsline{toc}{section}{Bibliography}
\printbibliography

\end{document}